\documentclass
[reqno]{amsart}
\usepackage[utf8]{inputenc}
\usepackage{amssymb,amscd}
\usepackage{thmtools}

\usepackage[textsize=tiny]{todonotes}
\usepackage{hyperref}
\usepackage[shortlabels]{enumitem}
\usepackage{float} 
\usepackage{cleveref} 
\usepackage{graphicx}
\usepackage{float}


\newtheoremstyle{alstandard}{7pt}{3pt}{\rm}{}{\scshape}{:}{0.5em}{}
\theoremstyle{alstandard}
\swapnumbers
\declaretheorem[name=Theorem]{theorem}
\numberwithin{theorem}{section}
\declaretheorem[sibling=theorem, name=Lemma]{lemma}
\declaretheorem[sibling=theorem, name=Proposition]{prop}

\declaretheorem[sibling=theorem, name=Definition]{defi}

\declaretheorem[sibling=theorem, name=Corollary]{cor}
\declaretheorem[sibling=theorem, name=Remark]{remark}

\newcommand{\R}{\mathbb{R}}
\newcommand{\C}{\mathbb{C}}
\newcommand{\N}{\mathbb{N}}

\newcommand\Z{\mathbb Z}
\renewcommand\P{\mathbb P}

\newcommand{\E}{\mathbb E}

\DeclareMathOperator{\id}{id}

\DeclareMathOperator{\GM}{GM}

\DeclareMathOperator{\apex}{ap}
\DeclareMathOperator{\smooth}{reg}

\renewcommand{\i}{\overset\circ\imath} 

\DeclareMathOperator{\GL}{GL}

\usepackage{pifont}
\begin{document}
	\title[Nondefectivity of reducible secant varieties]{On defectivity of joins, reducible secants and Fröberg's conjecture} 
	
	\author[Blomenhofer]{Alexander Taveira Blomenhofer}
	\address{QMATH (University of Copenhagen), Lyngbyvej 2, 2100 Copenhagen, Denmark}
	\email{atb@math.ku.dk}
	
	\author[Casarotti]{Alex Casarotti}
	\address{Università di Ferrara, Dipartimento di Matematica, Via Nicolò Machiavelli 30, 44121, Italy}
	\email{alex.casarotti@unife.it}

	\begin{abstract} 
		We introduce the notion of $ r $-defectivity for a vector bundle on a quasi-projective variety. 
		Using this tool, we prove several previously unknown cases of Fröberg's conjecture and also of the postulation problem for fat point schemes. Our techniques also allow us to study $ r $-nondefectivity for joins and secants of reducible varieties. As a consequence, we derive results for mixture distributions and partition ranks.   
	\end{abstract}

	\maketitle
	\raggedbottom
	
\section{Introduction}\label{sec:intro}
Nondefectivity is a widely studied notion in the context of secant varieties of an irreducible affine cone. In this work, we examine what can be said if one drops the requirement of irreducibility. More generally, we will consider \emph{$ V $-embedded vector bundles}, which associate a linear subspace of $ V $ to each point of a projective variety, and we introduce the notion of $ r $-defectivity for such bundles. 

Recall that, given an irreducible affine cone $ X $, a variety is called $ r $-\emph{nondefective}, if a sum of tangent spaces $ \langle T_{x_1} X,\ldots, T_{x_m} X\rangle $ at general points $ x_1,\ldots,x_r \in X$ has dimension $ r\cdot \dim X $. 

We aim to generalize the notion of $ r $-nondefectivity to reducible varieties $ X $ and to arbitrary vector bundles. The classical definition will correspond to the case where the vector bundle is the embedded tangent bundle of $ X $. This generalization is motivated by questions about Fröberg's conjecture, fat point schemes, mixture distributions and partition ranks. In the following, let us briefly describe each of the problems which motivated this work.

\medskip\noindent
\textbf{Fröberg's conjecture} \cite{froberg1985inequality} asks about the Hilbert series of ideals generated by general forms $ f_1,\ldots,f_m \in \C[x_1,\ldots,x_n] $. 
A seminal result of Nenashev \cite{nenashev2017note} allows to determine the Hilbert series in some cases, assuming that all generators $ f_1,\ldots,f_m $ are of the same degree $ d $. Nenashev's result is special in the sense that it allows the number $ m $ of forms to be relatively large in relation to the number $ n $ of variables. 
For forms of arbitrary degree, no similar result is known. 

\medskip\noindent
\textbf{The postulation problem for fat point schemes} asks to determine the Hilbert function of the ideal $ I_{p_1}^{m_1} \cap \ldots \cap I_{p_k}^{m_k} $ of forms vanishing on given points $ p_1,\ldots,p_k \in \mathbb \C^n $ to multiplicities $ m_1,\ldots,m_k \in \N $. The problem in general is open, even if some remarkable partial result has been proved in \cite{Postinghel_2023} and \cite{Postinghel_Dumitrescu}.

\medskip\noindent
The \textbf{partition rank} $r$ was introduced by \cite{Naslund_2020} as a generalization of slice rank. It corresponds to the smallest number $ r $ needed to write a tensor $ T $ as a sum of tensors $ t_1 + \ldots + t_r $, where each $ t_i $ splits according to a partition. The partition may be chosen freely out of a set $ \Lambda $ of allowed partitions (see \cite{Oneto2025} for a self-contained introduction to various interesting notions of rank). In this work we aim to study when a tensor of small partition rank has only finitely many decompositions of minimum rank. 

\medskip\noindent
The \textbf{parameter identifiability problem} for mixture distributions from moments asks, when the moments of a mixture distribution uniquely determine the parameters. E.g., one may ask whether a mixture of five Gaussian and seven Laplace distributions is uniquely determined by its moments of degree $ 5 $. 


\subsection{Techniques and reducible secants}
A common denominator of the abovementioned problems is that they can be treated in the framework of vector bundles. More precisely, they are questions about the (non)defectivity of $ V $-embedded vector bundles. We will describe in detail how to associate a vector bundle to each problem in \Cref{sec:applications}.

For now, let us focus on secants of reducible varieties. This framework will in particular encompass the problems of partition rank and mixture distributions. If a variety $ X $ has $ k $ irreducible components $ X = X_1\cup \ldots\cup X_k $, then its $ r $-secant variety $ \sigma_r(X) $ is the union of joins $ J(X_{i_1},\ldots,X_{i_r}) $, where $ i_1,\ldots,i_r\in \{1,\ldots,k\} $. Not all joins need to be irreducible components, but all irreducible components of $ \sigma_r(X) $ are joins. 
For joins as above, there exists a natural notion of expected dimension, which is $ \min\{\dim X_{i_1} + \ldots + \dim X_{i_r}, \dim \langle X \rangle\} $. A join is called defective, if its dimension is strictly less than the expected dimension. Unfortunately, there are not many general results on the dimensions of joins. However, we are able to provide results by looking simultaneously at all joins, which occur in the secant of a fixed algebraic variety. 
Intuitively, the reason for this is that as the number $ k $ of irreducible components of $ X $ is fixed and $ r $ grows large, the joins $ J(X_{i_1},\ldots,X_{i_r}) $ will contain large secant varieties of irreducible components of $ X $, i.e., they have to ``repeat'' irreducible components. 

In \Cref{cor:reducible-secants-nondefective}, we will derive a general criterion to determine when a (reducible) secant variety $ \sigma_r(X) $ is nondefective. 
Note that there are two reasonable notions of defectivity for a reducible secant variety, depending on whether one cares about all joins in $ \sigma_r(X) $ or only about the non-redundant joins. These will be defined in \Cref{def:geometrically-defective}.

\subsection{Organisation of the paper}
In \Cref{sec:stationarity-lemma-rectangular}, we prove the main result, \Cref{thm:k-factors}, which provides a criterion for a vector bundle to be nondefective. The key technical ingredient is the \emph{rectangular stationarity lemma}, see \Cref{lem:stationarity-lemma-rectangular}. This lemma tracks a certain space along a ``rectangular'' sequence of multi-indices in $ \N_0^n $ and provides a tool to establish nondefectivity of joins. 
For the special case where the vector bundle is the embedded tangent bundle of a variety, we obtain nondefectivity results for secant varieties. See \Cref{cor:reducible-secants-nondefective}. In \Cref{thm:diagonal-filling-criterion}, we give a criterion when a join is filling. 
In \Cref{sec:applications}, we discuss applications to Fröberg's conjecture (\Cref{{thm:froeberg}}), fat point schemes (\Cref{thm:fat-points}),  Partition ranks (\Cref{thm:partition-rank}) and mixture distributions (\Cref{thm:mixture-gaussian-laplace}). 

\subsection{Related work}
Fröberg's conjecture \cite{froberg1985inequality} is a long-standing open problem about the Hilbert function of an ideal $ I = (f_1,\ldots,f_r)$ generated by general forms of given degrees. Various partial results have been achieved, see for instance \cite{Migliore2003} where Fröberg's conjecture is linked to the Weak Lefschetz Property (WLP), and \cite{Migliore20031} where the case of $r=n+1$ forms is treated.
Particularly relevant for our work is the seminal result of Nenashev \cite{nenashev2017note}, who solved many cases of Fröberg's conjecture in the special case, where all forms $ f_1,\ldots,f_r $ have equal degree. Our \Cref{thm:froeberg} can be seen as a generalization of Nenashev's result to forms of non-equal degree. 

The postulation of zero dimensional fat point schemes is a classical open problem in algebraic geometry: It asks when the linear system 
of homogeneous polynomials of given degree 
that vanish to prescribed multiplicities $m(p_i)$ at the general points $p_1,\dots,p_r \in \C^n$, has the {expected} dimension.  
The case of double points, i.e. $m(p_i)=2$ for $i=1,\dots,r$, has been solved by Alexander-Hirschowitz in \cite{AH95}. For general multiplicities, the characterization of defective cases, i.e. linear systems not of the expected dimensions, is still widely open: various notions of obstructions for the expected dimensionality have been given in \cite{Postinghel_Dumitrescu}. For points in the plane, the famous SHGH-conjecture (see \cite{SHGH}) proposes an explicit geometric criterion for the failure of the expected dimensionality, see for instance \cite{Cools07}. Our \Cref{thm:fat-points} uses a new approach with respect to the classical techniques of postulation problems, with which we are able to exhibit many nondefective cases in large numbers of variables. 

The identifiability of Gaussian mixtures is widely studied both from the algebraic perspective (\cite{Amendola_Faugere_Sturmfels_2016}, \cite{Amendola_Ranestad_Sturmfels_2018}) and from the computational side (\cite{Moitra_Valiant_2010}, \cite{Ge_Huang_Kakade_2015}). More general types of mixture distributions are very interesting from a statistical perspective, but results are sparse. To the best of our knowledge, we provide the first general identifiability result, which applies to mixtures of \emph{distinct} distributions. As an example application, we show in \Cref{thm:mixture-gaussian-laplace} how our main result can be applied to show finite-to-one identifiability of a mixture of Laplacian and Gaussian distributions. 

Partition rank was introduced by Naslund \cite{Naslund_2020} as a generalization of slice rank (\cite{Tao_2016_SymmetricCapset}, \cite{Tao_2016_SliceRankNotes}) and is mostly studied from the combinatorial side. With our \Cref{thm:partition-rank}, we hope to encourage the study of partition ranks also from the viewpoint of secant varieties, nondefectivity and identifiability. 

The methodology of this paper generalizes ideas from our earlier work \cite{Blomenhofer_Casarotti_2023}, where we extend the \emph{stationarity lemma} proved therein (\cite[Section 3]{Blomenhofer_Casarotti_2023}) to the larger class of reducible varieties. The earlier work \cite{Blomenhofer_Casarotti_2023} in turn builds on ideas from Nenashev's proof of the equal-degree case of Fröberg's conjecture. Note that our results presented in \Cref{fig:froeberg} and \Cref{fig:fat-points} leave a small gap of cases, where we cannot prove nondefectivity of the respective bundles. This gap is a necessary limitation of general results, which apply to all varieties, and it occurs already in the setting of irreducible varieties, see \cite{Blomenhofer_Casarotti_2023}. In fact, it is well-understood which irreducible affine cones $ X $ have the largest number of defective ranks, see \cite{Adlandsvik_1988}. Complete classifications are very difficult, but they have recently been achieved for Segre-Veronese varieties of Veronese-degree at least $ 3 $, see \cite{Abo_Brambilla_Galuppi_Oneto_2024}. We note that our earlier work \cite{Blomenhofer_Casarotti_2023} was used as an ingredient for the complete classification in \cite{Abo_Brambilla_Galuppi_Oneto_2024}. This gives hope that future work might be able to close the aforementioned gaps related to Fröberg's conjecture and fat point schemes. 

	\section{Preliminaries}\label{sec:preliminaries}

Let us briefly recall the main notions used in this paper. Most of our varieties will be (quasi-)affine cones $ X $ in a linear space $ V $, i.e. $\mathbb{C} \cdot X \subseteq X$. The corresponding projective varieties we denote by $ \P(X) \subseteq \P(V) $.

Frequently, we will also have a group $ G $ acting on the space $ V $, giving it the structure of a $ G $-module. 
We say that a variety $ X $ is $ G $-\emph{invariant}, if for all $ v\in X $ and $ g\in G $, it holds $ gv\in X $. We use two different notions of irreducibility, that are not to be confused: A variety is called irreducible, if it cannot be covered by the union of two proper subvarieties. A $ G $-module $ V $ is called irreducible, if $ V $ is not the zero module and $ V $ does not contain any proper nonzero $ G $-submodules. The set of smooth points of a variety $ X $ is denoted $ X_{\smooth} $. Note that an affine cone is smooth if and only if its only singular point is the origin in the vector space $V$. For every smooth point $x \in X$ we denote by $T_x X$ the (embedded) affine tangent space to $X$ at $x$. If $x_1,\dots,x_r$ are smooth points of $X$ we denote by $\langle T_{x_1}X_1,\ldots, T_{x_r} X_r  \rangle$ the linear span of the corresponding tangent spaces.

\begin{defi}
	The \emph{$r$-th secant variety} $\sigma_r (X)$ of the affine cone $X$ is the Zariski closure of the set of vectors that can be spanned by $r$ points in $X$, i.e.
	\[	
		\sigma_r(X):=\overline{ \{x_1 + \ldots + x_r \mid x_1,\ldots,x_r \in X\}}
	\] 
	Similarly, the \emph{join} $ J(X_1,\ldots,X_r) $ of affine cones $ X_1,\ldots,X_r $ is defined as the closure of the set $ \{x_1 + \ldots + x_r \mid x_1\in X_1,\ldots,x_r \in X_r\} $. 
\end{defi}
The \emph{expected dimension} $ e(X_1,\ldots,X_r) $ of the join  $J(X_1,\ldots,X_r) $ is defined as $ \min\{\dim X_1 + \ldots + \dim X_r , \dim V\}$. If the dimension of $ J(X_1,\ldots,X_r) $ is smaller than expected, we say that $X$ is $r$-\emph{defective}. Otherwise, $ X $ is called $ r $-\emph{nondefective}.

Secant varieties are a special case of joins, as $ \sigma_r(X) = J(X,\ldots,X) $. If $ X_1,\ldots,X_r$ are irreducible, then so is the join $ J(X_1,\ldots,X_r) $. 
However, we will be focused on the secants where $ X = X_1\cup \ldots \cup X_k $ is reducible, with $ k $ irreducible components $ X_1,\ldots,X_k $. As briefly outlined in \Cref{sec:intro}, it holds that the secant of a reducible variety $ X $ decomposes as a union of joins of its irreducible components. Precisely, one has 
\begin{align}\label{eq:secant-reducible-join-decomp}
	\sigma_r(X_1\cup \ldots \cup X_k) = \bigcup_{i_1,\ldots,i_r = 1}^k J(X_{i_1},\ldots,X_{i_r})
\end{align}
Recall that after removing redundant joins, \eqref{eq:secant-reducible-join-decomp} gives an irreducible decomposition of $ \sigma_r(X) $. In general, it does not make sense to associate an expected dimension to a secant variety, as its irreducible components might have different dimensions. However, we can still define when a reducible variety is $ r $-defective. There are two reasonable notions of defectivity for a reducible secant variety, depending on whether one cares about all joins in $ \sigma_r(X) $ or only about the non-redundant joins.

\begin{defi}\label{def:geometrically-defective}
	Let $ X $ be a an affine cone with irreducible components $ X_1,\ldots,X_k $. 
	\begin{enumerate}
		\item $ X = X_1\cup \ldots\cup X_k $ is called \emph{$ r $-defective}, if any of the $ r $-joins $ J(X_{i_1},\ldots,X_{i_r}) $ is defective.
		\item We say that $ X = X_1\cup \ldots\cup X_k $ is \emph{geometrically $ r $-defective}, if there is an irreducible component $ J(X_{i_1},\ldots,X_{i_r}) $ of $ \sigma_r(X) $, which is defective. 
	\end{enumerate}
\end{defi}
Our result in \Cref{sec:stationarity-lemma-rectangular} will guarantee the stronger property of $ r $-nondefectivity, and hence also geometric $ r $-nondefectivity. 
The main tool in order to compute the dimension of joins and secant varieties dates back to Terracini \cite{Te12} and can be summarized as follows:
\begin{lemma}[{Terracini, see \cite{Te12}}]
	For general $ x_1 \in X_1,\ldots,x_r \in X_r $ and general $ z\in \langle x_1,\ldots,x_r \rangle $, the tangent space to the join $ J(X_1,\ldots,X_r) $ at  $ z $ equals $\langle T_{x_1}X_1,\ldots, T_{x_r} X_r  \rangle$. 

	Moreover, for smooth points $ x_1 \in X_1,\ldots,x_r \in X_r $, and any $ z\in \langle x_1,\ldots,x_r \rangle $, $\langle T_{x_1}X,\ldots, T_{x_r} X  \rangle$ is contained in the tangent to $ J(X_1,\ldots,X_r)$ at $ z $. 
\end{lemma}

\begin{prop}\label{prop:apex-space}
	The \emph{apex space}\footnote{The apex space is classically known as the vertex of an irreducible variety $X$, see for instance the seminal book \cite{zak1993tangents}. Here we use the apex notation in order to be consistent with the setting of our previous work \cite{Blomenhofer_Casarotti_2023}.}	
	 of an (irreducible) affine cone $ X $ is defined as
	\begin{align*}
		\apex(X) = \{x\in X \mid X+x = X\}.
	\end{align*}
	For brevity, we also denote $ \apex_r (X) := \apex(\sigma_r(X)) $. It holds that 
	\begin{align*}
		\apex_r(X) = \bigcap_{(x_1,\ldots,x_r) \in \mathcal{U}_{r, X} } \langle T_{x_1} X,\ldots, T_{x_r} X \rangle, 
	\end{align*}
	where $ \mathcal{U}_{X, r} $ denotes the set 
	\begin{align*}
		\{x\in X_{\mathrm{reg}}^r \mid \langle T_{x_1} X,\ldots, T_{x_r} X \rangle = T_{x_1 + \ldots + x_r} \sigma_r (X) \text{ has generic dimension} \}.
	\end{align*}
\end{prop}
\begin{proof} See \cite[Prop. 2.2]{Blomenhofer_Casarotti_2023}. 
\end{proof}

\begin{remark}
	The apex space is a linear space satisfying $ \apex(X) \subseteq X $. 
	If $ V $ is a $ G $-module and $ X $ is $ G $-invariant, then $ \apex(X) $ has a natural induced structure of $ G $-submodule of $ V $. 
\end{remark}

\subsection{Defective vector bundles and their apices}

\begin{defi}
	An embedded vector bundle is a quadruple $ (E, X, \pi, V) $ comprised of a quasi-affine cone $ X $, a finite-dimensional linear space $ V $, a Zariski-locally trivial algebraic map $ \pi\colon E\to X $ and a subset $ E\subseteq V $ such that $ \pi^{-1}(\{x\}) $ is a linear subspace of $ V $ for every $ x\in X $. We denote with $ E_x := \pi^{-1}(\{x\}) $ for brevity. The (vector) spaces $ E_x $ are called the \emph{fibers} of $ E $.
\end{defi}

In the context of varieties $X$ equipped with a $G$-action and vector bundles embedded into a $G$-module $V$, we give the following definition:

\begin{defi} Let $ E $ be a vector bundle on a $ G $-invariant variety $ X $, embedded into a $ G $-module $ V $. Then, $ E $ is called 
	$ G $-equivariant, if $ E_{gx} = g E_x $ for all $ g\in G $.
\end{defi}

\begin{prop} 
	If $ X = X_1\cup \ldots \cup X_k $ is the irreducible decomposition of $ X $, then the dimension of $ E_x $ is constant on each irreducible component. We call these dimensions $ (N_1,\ldots,N_k) $ the ranks of $ E $. Moreover if $X_i \cap X_j \neq \{0\}$ then $N_i=N_j$. 
\end{prop}
\begin{proof}
This is a standard general fact in algebraic geometry. For a proof see for instance \cite[III.5.8]{Hartshorne}
\end{proof}

\begin{defi}
	An embedded vector bundle $ E $ of rank $ N $ on an irreducible variety $ \P(X) $ is called $ r $-nondefective, if for general $ x_1,\ldots,x_r \in X $, the dimension of $ \langle E_{x_1},\ldots,E_{x_r} \rangle $ equals $ r N $. 
\end{defi}

Note that when $E$ is the embedded tangent bundle of $X$ then the notion of $r$-defectivity reduces to the classical one.

For reducible varieties, the definition is slightly more involved, as the correct dimension depends on how many $ x_i $ are chosen from each irreducible component. This motivates the following definition. 

\begin{defi}
	We say that a sequence of points $ x = (x_1,\ldots,x_r) $ with $ x_i\in X $ is of \emph{(component) type} $ \alpha \in \N_0^k $, if $ \# \{j \in \{1,\ldots,r\} \mid x_j\in X_i\} = \alpha_i $ and no $ x_i $ lies in the intersection of two components. 
\end{defi}
Note that if in the above definition, it holds $ |\alpha| = r $. We call a sequence $ x = (x_1,\ldots,x_r)$ in $ X $ general of type $ \alpha $, if each $ x_j $ is a general point of one of the $ X_i $.    

\begin{defi}
	Let $ E $ be a $ V $-embedded vector bundle on a variety $ \P(X) $ with components $ X = X_1\cup \ldots \cup X_k $ and respective ranks $ (N_1,\ldots,N_k) $. Let $ x_1,\ldots,x_r $ be a general sequence of type $ \alpha \in \N_0^k $ in $ X $. 
	\begin{enumerate}
		\item $ E $ is called $ \alpha $-\emph{nondefective}, if the dimension of $ \langle E_{x_1},\ldots,E_{x_r} \rangle  $ is equal to $ \alpha_1N_1 + \ldots + \alpha_k N_k $. 
		\item $ E $ is called $ \alpha $-\emph{filling}, if  $ \langle E_{x_1},\ldots,E_{x_r} \rangle = V$. 
		\item We say that $ E $ is \emph{of $ \alpha $-expected dimension}, if the dimension of $ \langle E_{x_1},\ldots,E_{x_r} \rangle $ equals $ \min\{\alpha_1N_1 + \ldots + \alpha_k N_k, \dim V\} $.
	\end{enumerate}
	Additionally, $ E $ is called $ r $-\emph{nondefective}, if $ E $ is $ \alpha $-nondefective for each $ \alpha \in \N_0^k $ with $ |\alpha| = r $. Note that $ E $ is of $ \alpha $-expected dimension, if $ E $ is either $ \alpha $-nondefective or $ \alpha $-filling.
\end{defi}

Let $ \mathcal{U}_{E, \alpha} =  \{x\in X^r \mid x \text{ has type } \alpha \text{ and } \langle E_{x_1},\ldots,E_{x_r} \rangle \text{ is of generic dimension} \} $. Note that the generic dimension of $ \langle E_{x_1},\ldots,E_{x_r} \rangle $ is at most $ \alpha_1 N_1 + \ldots + \alpha_k N_k $, but it might be less. The generic dimension is well-defined for sequences of type $ \alpha $.

\begin{defi}
	We define the $ \alpha $-apex of the bundle $ E $ to be
	\begin{align*}
		\apex_{\alpha}(E) :=  \bigcap_{(x_1,\ldots,x_m) \in \mathcal{U}_{E, \alpha}} \langle E_{x_1}, \ldots, E_{x_m} \rangle 
	\end{align*}
\end{defi}
The apex space is a technical tool for the proof of our main result. For the special case where $ E = T $ is the tangent bundle of an irreducible variety $ X $, the $ (r) $-apex space is classically known as the vertex of $\sigma_r(X)$, see for instance the seminal book \cite{zak1993tangents}. See also \cite[Preliminaries]{Blomenhofer_Casarotti_2023} for more detailed explanation. In the more general setting of irreducible varieties, it is (for technical reasons) useful to consider \emph{partial} apices, where the intersection runs only over part of the sequence, while other points are fixed. The precise definition follows. 

\begin{defi}\label{def:partial-apex}
	Let $x=(x_1,\dots,x_r)$ be a sequence of points in $ X $ and $ \alpha\in \N_0^n $. Write $ s = |\alpha| $. We define the partial $\alpha$-apex of the bundle $E$ at $ x $ to be
	\begin{align*}
		\apex_{\alpha}^{x}(E) :=  \bigcap_{(y_1,\ldots,y_s) \in \mathcal{U}_{E, \alpha}} \langle E_{x_1}, \ldots, E_{x_r}, E_{y_1},\ldots,E_{y_s} \rangle,
	\end{align*}
	where $ \mathcal{U}_{E, \alpha} $ denotes the set of $ \alpha $-type sequences $ y $ in $ X $, for which the space $ \langle E_{x_1}, \ldots, E_{x_r}, E_{y_1},\ldots,E_{y_s} \rangle $ has the generic dimension. 
\end{defi}
Note that the partial apex $ \apex_{\alpha}^{x}(E) $ depends on the points $x_{1},\dots,x_{r}$, which are fixed in the intersection. However, for general $ x $ of fixed component type $ \beta $, the dimension of $ \apex_{\alpha}^{x}(E) $ is of course constant. 
\begin{lemma}\label{lem:apex-lemma}
	Let $ \mathcal{U} $ a dense open subset of $ \mathcal{U}_{E, \alpha} $. Write $ m = |\alpha| $. Then,
	\begin{align*}
		\apex_{\alpha}(E) = \bigcap_{(x_1,\ldots,x_m) \in \mathcal{U}} \langle E_{x_1}, \ldots, E_{x_m} \rangle.
	\end{align*}
\end{lemma}
\begin{proof}
	The inclusion from left to right is clear. Now, let $ p $ be contained in the right hand side. Then, $ p $ lies in $ \langle E_{x_1}, \ldots, E_{x_r} \rangle  $ for all $ x = (x_1,\ldots,x_r) \in \mathcal{U} $. Let $ R $ denote the generic dimension of $ \langle E_{x_1}, \ldots, E_{x_r} \rangle  $. For all $ x\in \mathcal{U}_{E, \alpha} $, by the Plücker embedding, we can represent $ \langle E_{x_1}, \ldots, E_{x_r} \rangle  $ as a projective equivalence class of an element $ t(x_1,\ldots,x_r) \in \bigwedge^{R}(V) $, where $ t $ depends polynomially on $ (x_1,\ldots,x_r) $. Define $ f(p, x) = p\wedge t(x_1,\ldots,x_r) $. Clearly, $ f(p, x) $ is the zero tensor for all $ x\in \mathcal{U} $. By continuity, it is constantly zero. But for $ x\in \mathcal{U}_{E, \alpha} $, the identity $ p\wedge t(x_1,\ldots,x_r) = 0  $ implies that $ p\in  \langle E_{x_1}, \ldots, E_{x_r} \rangle $.
\end{proof}
\noindent
An analogous statement to \Cref{lem:apex-lemma} does of course hold for the partial apex.

\section{The rectangular stationarity lemma}\label{sec:stationarity-lemma-rectangular}
We consider an affine cone $ X $, which is embedded in some $ G $-module and which is closed under the $ G $-action. Let $ X = X_1\cup \ldots \cup X_k $ denote the decomposition of $ X $ into irreducible components. Furthermore, we consider a $ G $-module $ V $ and a $ G $-equivariant vector bundle $ \pi\colon E\to X $ on $ X $, such that all fibers $ E_x $ are embedded in $ V $. We denote the ranks of $ E $ on $ X_1,\ldots,X_k $ by $ N_1,\ldots,N_k $. 
Let $ \alpha \in \N_0^k $ and let $ x = (x_1,\ldots,x_{|\alpha|}) $ be a general sequence in $ X $ of type $ \alpha $. Let $ y\in X_i $. In this section, we will consider sequences of the form 
\begin{align*}
	a_{\alpha, i} := \dim \langle E_{x_1},\ldots,E_{x_{|\alpha|}} \rangle \cap E_{y}
\end{align*}
Note that $ 0 \le a_{\alpha, i} \le N_i$, and $ a_{\alpha, i} $ only depends on the choice of $ \alpha $ and $ i $, as the dimension is the same for all general sequences $ x $ of type $ \alpha $ and all general $ y\in X_i $. A bit more generally, we may also consider the sequences 
\begin{align}
	b_{\alpha, \beta, i} := \dim \apex_{\alpha}^{x}(E) \cap E_{y}, 
\end{align}
where $ x $ is a sequence in $ x $ of type $ \beta $, $ y\in X_i $ and $ \apex_{\alpha}^{x}(E) $ denotes the partial apex introduced in \Cref{def:partial-apex}. Clearly, we also have that $ 0 \le b_{\alpha, \beta, i} \le N_i$ and $  b_{\alpha, \beta, i}  $ only depends on the choice of $ \alpha, \beta\in \N_0^k $ and $ i\in \{1,\ldots,k\} $. 

\medskip
\noindent
We start with a technical fact, which makes an important statement about stationarity points of the sequences $ a_{\alpha, i} $ and $ b_{\alpha, \beta, i} $. Note that both sequences are monotonic, in the sense that if $ \alpha\preceq \alpha' $ and $ \beta\preceq \beta' $, then $ a_{\alpha, i} \le a_{\alpha', i}$ and $ b_{\alpha, \beta, i} \le b_{\alpha', \beta', i}$. Here, $ \preceq $ denotes the entrywise semiordering of multi-indices. Via the stationarity lemma, we will establish that the sequences are ``usually'' either constantly zero, or they are strictly monotonically increasing. In other words, stationarity points are ``rare''.  
The stationarity lemmata proved in this work are a generalization of \cite{Blomenhofer_Casarotti_2023}, where it is showed that for irreducible varieties $ X $ (corresponding to $ k = 1 $), the only stationarity points can be at ``$ a_{re_1, 1} = 0 $'' or ``$ a_{re_1} = N_1 $'', where $ e_1 = (1, 0,\ldots,0) $ denotes the first standard basis vector. For higher $ k $, the statement is more involved. We develop it in the following. 

\medskip
\noindent
Fix $ j = 1,\ldots,k $ and consider the sequence $ b_s := b_{\alpha, \beta + s e_j, i }$. Note that $ 0\le b_{0} \le b_{1} \le\ldots $, since clearly $ \apex_{\alpha}^{x}(E) $ can only get bigger when adding $ s $ additional points from $ X_j $ to the sequence $ x $. W.l.o.g., assume that $ \beta_j = 0 $, as otherwise we can shift the sequence $ b_s $. We now state the stationarity lemma for partial apices. 

\begin{lemma}[Stationarity Lemma]\label{lem:stationarity-lemma-rectangular}
	Let $ z_1,\ldots,z_s \in X_j$ be general points and let $ x = (x_{1},\ldots,x_{r}) $ be a sequence of generic points of type $ \beta $, where $ r = |\beta| $. If $ 0\ne b_s = b_{s+1} $, then $ \apex_{\alpha}^{(x, z)}(E) \cap E_y = \apex_{\alpha+se_j}^{x}(E) $. 
\end{lemma}
\begin{proof}
	Let $ y $ be a general point in $ X_i $. We define 
	\begin{align*}
		L(x, z)  &= \apex_{\alpha}^{(x, z)}(E)
	\end{align*}
	Consider another generic sequence $ z' = (z_1',\ldots,z_s') \in X_j^s$. \\
	\textbf{Step 1:} We are to show that $ L(x, z) \cap E_{y} = L(x, z') \cap E_{y} $. \\
	Indeed, observe that $ L(x, z) \cap E_y$ is a subspace of $ L(x, (z, z_1')) \cap E_y $ and both have the same dimension: Indeed, $  L(x, z) \cap E_y $ has dimension $ b_s $ by assumption and $ L(x, (z, z_1')) \cap E_y  $ has dimension $ b_{s+1} $, since $ (z, z_1') $ is a general sequence of $ s+1 $ points in $ X_j $. Since $ b_s = b_{s+1} $, therefore, the two spaces must be equal. 
	
	Continuing with the same argument, we also have that $ L(x, (z_1', z_2,\ldots,z_s)) \cap E_y $ is a subspace of $ L(x, (z, z_1')) \cap E_y $ and both have the same dimension. Summarized, we showed that
	\begin{align*}
			L(x, z) \cap E_y = L(x, (z, z_1')) \cap E_y = L(x, (z_1', z_2,\ldots,z_s)) \cap E_y
	\end{align*}
	In other words, we showed that the intersection $ L(x, z)\cap E_y $ does not change, if we switch out $ z_1 $ by another general point $ z_1' $ from $ X_j $. By repeating the argument inductively, we can switch $ z_1,\ldots,z_s $ with $ z_1',\ldots,z_s' $ and consequentially we see that $ L(x, z)\cap E_y = L(x, z')\cap E_y $. 

	\textbf{Step 2:} We showed in Step 1 that there is a Zariski dense open subset $ \mathcal{U}\subseteq X_j^k $ such that for all $ z'\in \mathcal{U} $, $ L(x, z) \cap E_y = L(x, z')\cap E_y$. This implies that   
	\begin{align*}
		L(x, z) \cap E_y = \bigcap_{(z_1',\ldots,z_s') \in \mathcal{U}} L(x, z') \cap E_y = \apex_{\alpha+se_j}^{x}(E).
	\end{align*}
	This is exactly the statement that was claimed. 
\end{proof}

The proof of our main result is obtained by induction over the number $ k $ of irreducible components of $ X $. The case $ k = 1 $ was proven in \cite{Blomenhofer_Casarotti_2023}. 
For the reader's convenience, we give a separate proof for the case where $ X = X_1 \cup X_2 $ has two irreducible components in the following \Cref{thm:2-factors}. This is not strictly necessary, since one could also use the main result of \cite{Blomenhofer_Casarotti_2023} as the base case for the induction. However, it makes the proof more digestible. The general induction will be given thereafter in \Cref{thm:k-factors}. 

\begin{theorem}\label{thm:2-factors}
	Let $ V $ an irreducible $ G $-module and $ E $ a $ G $-invariant $ V $-embedded vector bundle on $ X $, where $ X $ has components $ X_1, X_2 $ and $ E $ has rank $ N_i $ on $ X_i $ for $ i = 1,2 $ and $ N_1\ge N_2 $.  Assume that $ \alpha\in \N_0^2 $ is such that $ \alpha_1 < \frac{\dim V}{N_1} - N_1 $ and 
	\begin{align}\label{eq:dimension-constraint-two-components}
		\alpha_1 N_1 + \alpha_2 N_2 + N_1 N_2 < \dim V. 
	\end{align}
	Then, $ E $ is $ \alpha $-nondefective. 
\end{theorem} 
\begin{proof}
	Observe that $ E $ cannot be $ (\alpha_1, 0) $-defective: Indeed, this is covered by an earlier result on bundles on irreducible varieties (see \cite[Corollary 3.5]{Blomenhofer_Casarotti_2023}). 
	Therefore, there must exist $ s_0 \in \N_0 $ such that $ E $ is $ (\alpha_1, s_0) $-nondefective, but $ E $ is $ (\alpha_1, s_0+1) $ defective. As the defect is caused by adding a space $ E_y $, where $ y\in X_2 $, we have that $ a_{(\alpha_1, s_0), 2} \ne 0 $. In other words, $ \langle E_{x_1},\ldots,E_{x_{\alpha_1}}, E_{y_1},\ldots,E_{y_{s_0}} \rangle \cap E_y \ne \{0\}$ for all general $ x_1,\ldots,x_{\alpha_1}\in X_1 $ and general $ y_1,\ldots,y_{s_0}, y\in X_2 $. Consider now the sequence of values $ a_s := a_{ (\alpha_1, s), 2} $ for $ s\ge s_0 $. 
	Assume that $ s $ is a stationarity point of this sequence, i.e., $ a_s = a_{s+1} $. 
	Then, we obtain from \Cref{lem:stationarity-lemma-rectangular} that the partial apex, $\apex_{(0, s)}^{x}(E)$, as introduced in \Cref{def:partial-apex}, has a nontrivial intersection with $ E_y $: Indeed, \Cref{lem:stationarity-lemma-rectangular}, applied with ``$ i = j = 2 $'', ``$ (z_1,\ldots,z_s) = (y_1,\ldots,y_{s})$'' gives us that
	\begin{align}\label{eq:ExEx-partial-apex}
		\{0\} \ne \langle E_{x_1},\ldots,E_{x_{\alpha_1}}, E_{y_1},\ldots,E_{y_{s}} \rangle \cap E_y = \apex_{(0, s)}^{x}(E) \cap E_y, 
	\end{align}
	This shows that the partial apex, obtained from intersecting in \eqref{eq:ExEx-partial-apex} over all choices of $ y_1,\ldots,y_s $, intersects $ E_y $ in a nonzero space. (In particular, the partial apex $ \apex_{(0, s)}^{x}(E) $ is nonzero). 
	Define for $ t\in \N_0 $ the dimension
	\begin{align*}
		b_t := \dim \apex_{(0, s)}^{(x_1,\ldots,x_{\alpha_1+t})}(E) \cap E_y. 
	\end{align*}
	of the space on the right hand side for general $ x_1,\ldots,x_{\alpha_1+t}\in X_1 $. Note that this dimension is well-defined, i.e., it does not depend on the (general) choice of the $ t $ additional points from $ X_1 $. Also, note that $ b_0 = a_s $. 
	Furthermore, observe that $ b_0\le b_1\le b_2\ldots $, as the space $ \dim \apex_{(0, s)}^{(x_1,\ldots,x_{\alpha_1+t})}(E) \cap E_y.  $ clearly can only get larger by increasing $ t $. 
	Now, if $ t $ is a stationarity point of the sequence $ b_1\le b_2 \le \ldots $, then $ b_t = b_{t+1} $. Using again \Cref{lem:stationarity-lemma-rectangular}, we can show that 
	\begin{align*}	
		 \dim \apex_{(0, s)}^{(x_1,\ldots,x_{\alpha_1+t})}(E) \cap E_y = \apex_{(\alpha_1+t, s)}(E) \cap E_y. 
	\end{align*}
	In particular, we showed that in fact that the full, non-partial apex  $ \apex_{(\alpha_1+t, s)}(E) $ is nonempty. By \Cref{lem:apex-lemma}, the (full) apex is a $ G $-module. Since $ V $ is an irreducible $ G $-module, the apex can only be nonzero, if it equals $ V $. However, if $  \apex_{(\alpha_1+t, s)}(E) = V $, then in particular, $ E $ is $ (\alpha_1+t, s) $-filling. Therefore, the sequence $ (b_t)_{t\ge 0} $ is strictly monotonic, until $ E $ is $ (\alpha_1+t, s) $-filling.
	
	Let us now prove the statement of the theorem: To the contrary, assume that $ E $ was $ \alpha $-defective. Note that due to the dimension constraint, \eqref{eq:dimension-constraint-two-components}, $ E $ cannot be $ (\alpha + se_2) $-filling for any $ s\le N_2 $. The sequence $ s\mapsto a_{\alpha + se_2} $ has to be either strictly increasing for all $ 0 \le s \le N_2-1 $ or there will be a stationarity point $ \hat{s} $. In the first case, we see that $ E $ must be $ (\alpha+N_2-1) $-filling, which contradicts the dimension constraint, \eqref{eq:dimension-constraint-two-components}. 
	
	In the second case, we obtain a sequence $ t\mapsto b_t $ as described above. Observe that $ a_{\hat{s}} = b_0 $. Additionally, the sequence $ t\mapsto b_t $ cannot have a stationary point unless $ E $ is $ (\alpha+(t, \hat{s})) $-filling. Due to the dimension constraint from \eqref{eq:dimension-constraint-two-components}, it cannot be filling, if $ \hat{s}+t \le N_2 $ (since $ N_1\ge N_2 $, adding $ N_2 $ spaces can increase the expected dimension by at most $ N_1 N_2 $). Consequently, both sequences $ s\mapsto a_{\alpha + se_2} $ and $ t\mapsto b_t $ are strictly increasing, as long as $ s\le \hat{s} $ and $ t\le N_2-1-\hat{s} $. In particular, this means that $ a_{\hat{s}} \ge \hat{s}+1 $ and $ b_t \ge t+\hat{s}+1  $, if $ t\le N_2-1-\hat{s} $. 
	Therefore, we see that $ b_{N_2-1-\hat{s}} = N_2 $ and thus $ E $ has to be $ (\alpha + (\hat{s},  N_2-1-\hat{s})) $-filling. 
	This is a contradiction: $ E $ cannot be $ (\alpha + (\hat{s},  N_2-1-\hat{s})) $-filling  due to the dimension constraint, \eqref{eq:dimension-constraint-two-components}. 
	We conclude that $ E $ must be $ \alpha $-nondefective. 
\end{proof}

\noindent
The previous result may be generalized to $ k $ irreducible components as follows.

\begin{theorem}\label{thm:k-factors}
	Let $ V $ an irreducible $ G $-module and $ E $ a $ G $-invariant $ V $-embedded vector bundle on $ X $, where $ X $ has components $ X_1,\ldots,X_k $ and $ E $ has rank $ N_i $ on $ X_i $ for $ i = 1,\ldots,k $, w.l.o.g. ordered such that $ N_1\ge \ldots \ge N_k $.  Assume that $ \alpha\in \N_0^2 $ is such that 
	\begin{align}\label{eq:dimension-constraint-k-components}
		&\alpha_1 N_1 + N_1(N_1-1) < \dim V\\
		&\alpha_1 N_1 + \alpha_2 N_2 + N_1(N_2-1) < \dim V\nonumber\\
		&\qquad\qquad\qquad\vdots\nonumber \\
		&\alpha_1 N_1  + \ldots + \alpha_k N_k + N_1 (N_k-1) < \dim V.\nonumber  
	\end{align}
	Then, $ E $ is $ \alpha $-nondefective. 
\end{theorem}
\begin{proof}
	Assume to the contrary that $E$ is $\alpha$-defective for some
	\begin{align*}
	\alpha=(\alpha_1,\ldots,\alpha_k)\in\mathbb{N}_0^k
	\end{align*}
	satisfying the chain of inequalities \eqref{eq:dimension-constraint-k-components}.
	By the induction hypothesis (applied to the union $X_1\cup\cdots\cup X_{k-1}$ with ranks $N_1\ge\cdots\ge N_{k-1}$), $E$ is $(\alpha_1,\ldots,\alpha_{k-1},0)$-nondefective. Hence there exists a minimal $s_0\in\mathbb{N}_0$ such that $E$ is $(\alpha_1,\ldots,\alpha_{k-1},s_0)$-nondefective but $(\alpha_1,\ldots,\alpha_{k-1},s_0+1)$-defective. Equivalently, 
	\begin{align*}
		a^{(k)}_{s} := a_{(\alpha_1,\ldots,\alpha_{k-1},s_0+s),\,k} \neq 0
	\end{align*}
	for $s\ge 0$. In other words, if we fix a general sequence $ x $ in $ X $ of type $ (\alpha_1,\ldots,\alpha_{k-1}, 0) $ and general $ y_1,\ldots,y_{s_0+s}\in X_k $, then
	\begin{align*}
		\langle E_{x_1},\ldots,E_{x_{r}}, E_{y_1},\ldots,E_{y_{s_0+s}}\rangle\cap E_y\neq\{0\}
	\end{align*}
	for $ s\ge 0 $. Here, we denote $ r := \alpha_1 + \ldots + \alpha_{k-1} $. 
	
	Consider the sequence $s\mapsto a^{(k)}_s$ for $s\ge 0$. Since $ 0 < a^{(k)}_{s} \le N_k $, we know that there must exist a stationarity point in the interval $ [s_0, s_0 + N_k-1] $. 
	If $\hat{s}$ is a stationarity point, so that $a^{(k)}_{\hat{s}}=a^{(k)}_{\hat{s}+1}$, then from \Cref{lem:stationarity-lemma-rectangular} we obtain a non-trivial \emph{partial apex}. Precisely, we obtain 
	\begin{align}\label{eq:first-partial-apex-k}
		\langle E_{x_1},\ldots,E_{x_{r}}, E_{y_1},\ldots,E_{y_{s_0+\hat{s}}}\rangle\cap E_y = \apex^{x}_{(\hat{s}+s_0)e_k}(E) \cap E_y \neq \{0\} 
	\end{align}
	for general $y\in X_k$. Now, let us split off those parts of the sequence $ x $, which lie in $ X_{k-1} $. Precisely, reorder $ x $ such that $ x = (x', z) $, where $ z_1,\ldots,z_{\alpha_{k-1}} $ are the  general points from $ x $ lying in $ X_{k-1} $ and where $ x' $ is a general sequence of type $ (\alpha_1,\ldots,\alpha_{k-2}, 0, 0) $. For $t\in\mathbb{N}_0$, we choose $ z_{\alpha_{k-1}+1},\ldots,z_{\alpha_{k-1} + t} $ as $ t $ additional general points from $ X_{k-1} $ and we define
	\begin{align*}
		a^{(k-1)}_t := \dim\bigl(\apex^{(x', z_1,\ldots,z_{\alpha_{k-1} + t})}_{(\hat{s}+s_0)e_k}(E) \cap E_y \bigr) > 0 
	\end{align*}  
	It follows from equation \eqref{eq:first-partial-apex-k} that $ a^{(k-1)}_0 = a^{(k)}_{\hat{s}} $. Furthermore, the sequence $t \mapsto a^{(k-1)}_t$ is nondecreasing, so it is either strictly increasing or it has a stationarity point at some $\hat{t}$.  Again by \Cref{lem:stationarity-lemma-rectangular} if $\hat{t}$ is a stationarity point, then we get a second non-zero partial apex 
	\begin{align*}
		\apex^{x'}_{(\alpha_{k-1}+\hat{t})e_{k-1}+(\hat{s}+s_0)e_k}(E) \cap E_y \neq 0,
	\end{align*}
	where $ x' $ is a sequence in $X_1 \cup \dots \cup X_{k-2}$. We continue the process by splitting off all points in $ x' $, which lie in $ X_{k-2} $. Precisely, reorder $ x' $ such that $ x' = (x'', w) $. Now, consider the new sequence $u \mapsto a_{u}^{(k-2)}$, where 
	\begin{align*}
		a_{u}^{(k-2)}:=\dim\bigl(\apex^{(x'', w_1,\ldots,w_{\alpha_{k-2} + u})}_{(\alpha_{k-1}+\hat{t})e_{k-1}+(\hat{s}+s_0)e_k}(E)\cap E_y\bigr) > 0 
	\end{align*}
	Inductively, we obtain a sequence of stationarity points $ \hat{s}, \hat{t}, \hat{u}, \ldots $. For the sake of simplicity, let us rename them to $ \hat{s}_1,\ldots,\hat{s}_k $, where $ \hat{s}_1 = \hat{s}, \hat{s}_2 = \hat{t}, \hat{s}_3 = \hat{u}$. Note that the sequence
	\begin{align}\label{eq:staircase-seq-increasing}
		a_{0}^{(k)}<\ldots < a_{\hat{s}_1}^{(k)} = a_{0}^{(k-1)} < \ldots < a_{\hat{s}_2}^{(k-1)} < \ldots < a_{\hat{s}_{k-1}}^{(2)} =  a_{0}^{(1)} < \ldots < a_{\hat{s}_{k}}^{(1)} 
	\end{align} 
	is strictly increasing. Furthermore, it holds that $ a_{\hat{s}_{k}}^{(1)} = N_k$. Indeed, as $ \hat{s}_{k} $ is a stationarity point of $ s\mapsto a^{(1)}_s $, the full apex $ \apex_{\alpha + (\hat{s}_k,\ldots,\hat{s}_1)}(E) $ is nonempty by \Cref{lem:stationarity-lemma-rectangular} and so $ E $ must be $ \alpha + (\hat{s}_k,\ldots,\hat{s}_1) $-filling. 
	Counting the number of steps in the sequence \eqref{eq:staircase-seq-increasing}, this shows that $ 1 + \hat{s}_1 + \hat{s}_2 + \ldots + \hat{s}_k \le N_k $.

	Recall that we want to show that $ E $ is $ \alpha $-nondefective. To the contrary, we assumed that $E$ was $(\alpha_1,\dots,\alpha_k)$-defective. We showed that there exists a sequence $ \hat{s}_1,\ldots,\hat{s}_k $ summing up to at most $ N_k-1 $ such that $ E $ is $ \alpha + (\hat{s}_1,\ldots,\hat{s}_k)  $-filling. A dimension count shows that then we must have 
	\begin{align*}
		(\alpha_1 + \hat{s}_1) N_1  + \ldots + (\alpha_k + \hat{s}_k) N_k \ge \dim V. 
	\end{align*}
	Estimating $ \hat{s}_i N_i\le \hat{s}_i N_1 $ for all $ i=1,\ldots,k $ and using $  \hat{s}_1 + \hat{s}_2 + \ldots + \hat{s}_k \le N_k-1 $ yields 
	\begin{align*}
		\alpha_1N_1  + \ldots + \alpha_kN_k + N_1 (N_k-1) \ge \dim V. 
	\end{align*}
	However, this contradicts the last inequality in \eqref{eq:dimension-constraint-k-components}. Therefore, $ E $ must be $ \alpha $-nondefective. 
\end{proof}

\subsection{Secant varieties} We obtain the following consequence for secant varieties. 

\begin{cor}\label{cor:reducible-secants-nondefective}
	Let $ X $ be a $ G $-invariant affine cone in the irreducible $ G $-module $ V $, with $ k $ irreducible components of maximum dimension $ N_{\max} $. Then,  
	\begin{enumerate}
		\item $ X $ is $ r $-nondefective for all $ r\leq \frac{\dim V}{N_{\max}} - N_{\max} $.
		\item If $ r \geq \frac{\dim V}{N_{\max}} + N_{\max} $, then $ \sigma_r(X) = V $. 
	\end{enumerate}
\end{cor}
\begin{proof}
	Denote by $ X_1,\ldots,X_k $ the irreducible components of $ X $ of respective dimensions $ N_1,\ldots,N_k $, sorted such that $ N_1\ge \ldots \ge N_k $. 
	By Terracini's lemma, we have that $ X $ is $ r $-nondefective if and only if for all $ x_1,\ldots,x_r\in X $ of any possible component type $ \alpha\in \N_0^k $ (with $ |\alpha| = r $), the space $ \langle T_{x_1}X,\ldots,T_{x_r}X \rangle$ is of dimension $ \alpha_1N_1 + \ldots + \alpha_kN_k $. Hence, $ r $-nondefectivity of $ X $ is equivalent to $ \alpha $-nondefectivity of the tangent bundle $ T $ of $ X $ for each $ \alpha $ with $ |\alpha| = r $.
	Taking $ E = T $ as the tangent bundle in \Cref{thm:k-factors}, we see that $ T $ is $ \alpha $-nondefective for each $ \alpha $ such that $ \alpha_1N_1 + \ldots + \alpha_kN_k - N_{1}(N_{1}-1) < \dim V$. Indeed, this follows from estimating all terms $ N_i(N_1-1) $ by $ N_1(N_1-1) $. Furthermore, by estimating $ \alpha_iN_i\le \alpha_i N_1 $ we get that $ T $ is $ \alpha $-nondefective, if 
	\begin{align*}
		(\alpha_1 + \ldots + \alpha_k) + (N_{1}-1) < \frac{\dim V}{N_1}
	\end{align*}
	Clearly, this condition is the same for all $ \alpha $ with $ r = |\alpha| $. Using that $ N_1 = N_{\max} $, we conclude that $ X $ is $ r $-nondefective for all $ r\leq \frac{\dim V}{N_{\max}} - N_{\max} $.
	
	Conversely, let $ r \geq \frac{\dim V}{N_{\max}} + N_{\max} $. By \cite[Corollary 3.5]{Blomenhofer_Casarotti_2023}, we know that $ \sigma_{r}(X_1) $ is filling the space $ V $. Since $ X \supseteq X_1 $, we conclude that $ \sigma_r(X) = V $. 
\end{proof}

\subsection{Filling criterion and the diagonal stationarity lemma} 
\Cref{cor:reducible-secants-nondefective} gives a criterion when a reducible secant is filling, via its irreducible component of largest dimension. However, it is more interesting to have a criterion when a specific join is filling. We can in fact give a criterion when a bundle is $ \alpha $-filling. For irreducible varieties $ X $, such a filling criterion is trivially obtained as a byproduct of the argument for nondefectivity, see \cite{nenashev2017note} and \cite{Blomenhofer_Casarotti_2023}.
On the other hand, a direct criterion ensuring that a bundle $E$ on a \textit{reducible} variety $X=X_1 \cup \dots \cup X_k$ is $\alpha$-filling is not easily feasible from \Cref{thm:k-factors}. To overcome this, we state a different type of stationarity lemma, \Cref{lem:stationarity-lemma-diagonal}, and use it to give bounds for the $\alpha$-filling property.
\begin{theorem}\label{thm:diagonal-filling-criterion}
	Let $ V, G, E, X $ and $ N_1\ge\ldots \ge N_k $ as in \Cref{thm:k-factors}. Assume $ \alpha\in \N_0^n $ is such that $ \alpha_i\ge N_1 $ for each $ i = 1,\ldots,k $ and that 
	\begin{align}\label{eq:dimension-constraint-diagonal-filling}
		\alpha_1N_1 + \ldots + \alpha_k N_k > \dim V + (N_1-1)(N_1 + \ldots + N_k)
	\end{align}
	Then, $ E $ is $ \alpha $-filling. 
\end{theorem}
\begin{proof}
	From the diagonal stationarity lemma, see \Cref{lem:stationarity-lemma-diagonal} below, we conclude that there exist at most $ N_1-1 $ values of $ s\in \Z $ such that $ \alpha_s :=  \alpha + (s,\ldots,s) $ lies in $ \N_0^k $ and $ E $ is not of $ \alpha_s $-expected dimension.\footnote{That is, $ E $ is $ \alpha_s $-defective but not $ \alpha_s $-filling. }
	Due to the dimension constraint \eqref{eq:dimension-constraint-diagonal-filling}, $ E $ must be $ \alpha_s $-defective for each $ s\in \{0, -1, \ldots, -(N_1-1)\} $. This means that if $ E $ was not $ \alpha $-filling, then there would be $ N_1 $ cases on the line $ s\mapsto \alpha_s $, which are not of expected dimension, which contradicts \Cref{lem:stationarity-lemma-diagonal}. We conclude that $ E $ is $ \alpha $-filling. 
\end{proof}

\noindent
For $ m\in \N_0 $, let us denote $ \underline{m} = (m,\ldots,m) \in \N_0^k $. 
\begin{lemma}[Diagonal Stationarity Lemma]\label{lem:stationarity-lemma-diagonal}
	Let $ V $ be an irreducible $ G $-module and $ E $ be a $ G $-invariant $ V $-embedded vector bundle on $ X $. Let 
	\begin{align}
		a_{\alpha, i} := \dim \langle E_{x_1},\ldots,E_{x_r} \rangle \cap E_y, 
	\end{align}
	where $ x $ is a general sequence in $ X $ of type $ \alpha $ and $ y\in X_i $. If $ \alpha \in \N_0^k, i \in \{1,\ldots,k\}$ are such that $ a_{\alpha+\underline{1}, i} = a_{\alpha, i} \ne 0 $, then $ E $ is $ \alpha $-filling. In particular, we then have $ a_{\alpha, i} = N_{i} $. 
\end{lemma}
\begin{proof}
	For any sequence $ x = (x_{1},\ldots,x_r) $ of generic points of type $ \alpha $, let us denote 
	\begin{align}
		L(x)  &= \langle E_{x_1},\ldots,E_{x_r} \rangle
	\end{align}
	Consider another sequence $ x' = (x_1',\ldots,x_r') $ of generic points of type $ \alpha $ and pick yet another generic point $ y\in X_i $. 
	We are to show that $ L(x) \cap E_{y} = L(x') \cap E_{y} $. To this end, let us construct some extra sequences. 
	Without loss of generality, let us assume that $ \alpha\succeq \underline{1} $. Indeed, if $ \alpha_i = 0 $ for some $ i \in \{1,\ldots,k\} $, then we may ignore the existence of the $ i $-th component for the remainder of the proof. 
	Order $ x $ and $ x' $ such that $ x_1, x_1'\in X_1,\ldots, x_k,x_k' \in X_k $. This is possible, since $ \alpha\succeq \underline{1} $. Then, define 
	\begin{align}
		x^{\uparrow } &= (x_1,\ldots,x_r, x_{1}',\ldots,x_{k}'), \text{ and }\\
		x^{\updownarrow } &= (x_{1}',\ldots,x_{k}', x_{k+1},\ldots,x_r).
	\end{align}
	In other words, $ x^{\uparrow} $ is the sequence $ x $ augmented by the first $ k $ entries of $ x' $, while $ x^{\updownarrow} $ is the sequence $ x $, except that the first $ k $ entries of $ x $ are exchanged by the corresponding entries of $ x' $. Note that $ x^{\uparrow} $ is of type $ \alpha+\underline{1} $, while $ x^{\updownarrow} $ is of type $ \alpha $. Since $ a_{\alpha, i} = a_{\alpha+\underline{1}, i} $, we know that
	\begin{align}
		\dim L(x) \cap E_{y} = a_{\alpha, i} = a_{\alpha+\underline{1}, i} = \dim L(x^{\uparrow })\cap E_{y}. \nonumber
	\end{align}
	The space on the left-hand side is a subspace of the space on the right-hand side, of same dimension. Hence, they are equal: 
	\begin{align}
		L(x) \cap E_{y} = L(x^{\uparrow}) \cap E_{y}. 
	\end{align}
	Notice that the sequence $ x^{\uparrow } $ not only is an augmentation of $ x $, but also of $ x^{\updownarrow } $. Therefore, we may apply the same argument with $ x^{\updownarrow} $ instead of $ x $ to obtain
	\[
	{L}(x^{\updownarrow }) \cap E_{y} = 	L(x^{\uparrow }) \cap E_{y}. 
	\]
	Together, this implies
	\begin{align}\label{eq:ts-xm-exchange}
		L(x) \cap E_{y} = {L}(x^{\updownarrow}) \cap E_{y}
	\end{align}
	We thus showed that swapping an element from each irreducible component between $ x $ and $ x' $ does not change the space from \Cref{eq:ts-xm-exchange}. 
	Repeating this argument inductively with different indices, we obtain
	\begin{align}\label{eq:ts-exchange-all}
		L(x) \cap E_{y} = L(x') \cap E_{y}
	\end{align}
	In other words, any point $ p $ in $ L(x) \cap E_{y} $ lies in $L(x')$ for all general choices of $ x' $. 
	Thus,
	\begin{align}
		L(x) \cap E_{y} = \apex_{\alpha}(E) \cap E_y,
	\end{align}
	by \Cref{lem:apex-lemma}. 
	In particular, the $ \alpha $-apex of $ E $ is not the zero space, since $ a_{\alpha, i} > 0 $. The space $ \apex_{\alpha}(E) $ is $ G $-invariant and $ V $ is an irreducible $ G $-module. Therefore, it follows that $ V = \apex_{\alpha}(E) \subseteq L(x) $. Consequently, $ L(x) = V $ and thus $ a_{\alpha, i} = \dim E_y = N_i $. 
\end{proof}

Note that the diagonal stationarity lemma can of course also be used to give a numerical criterion for $ \alpha $-nondefectivity. However, the criterion is strictly weaker than \Cref{thm:k-factors}. Therefore, we refrain from stating it. 

\begin{remark}\label{rem:alpha-filling-when-a-maximal}
	If $ a_{\alpha-e_i,i} = N_i $, then $ E $ is $ \alpha $-filling. Indeed, if $ \langle E_{x_1},\ldots,E_{x_r} \rangle \cap E_y = E_y $ for all general $ y\in X_i $ and all $ x_1,\ldots,x_r\in X $ of component type $ \alpha $, then clearly, $ L(x) :=  \langle E_{x_1},\ldots,E_{x_r} \rangle $ contains the span of all spaces $ E_y $, which is $ G $-invariant, since $ E $ is $ G $-equivariant. As $ V $ is an irreducible $ G $-module, the span of all spaces $ E_y $ must thus be $ V $. Therefore, $ L(x) = V $. 
\end{remark}

\section{Applications}\label{sec:applications}

\subsection{Fröberg's conjecture for forms of non-equal degree}
Recall that the Hilbert function $ h_I $ of a homogeneous ideal $ I $ is given by $ h_{I}(\ell) := \dim I_{\ell} $, where $ I_{\ell} $ denotes the $ \ell $-th graded component of $ I $.
Given general forms $ f_1,\ldots,f_r $ of degrees $ (\underbrace{d_1,\ldots,d_1}_{\alpha_1},\ldots,\underbrace{d_k,\ldots,d_k}_{\alpha_k}) $, we consider the ideal $ I = (f_1,\ldots,f_r) $ generated by $ f_1,\ldots,f_r $. That is, the sequence $ f_1,\ldots,f_r $ contains $ \alpha_i$ forms of degree $ d_i $ for each $ i\in \{1,\ldots,k\} $. For $ \alpha\in \N_0^k $ with $ |\alpha| = r $, we call such a sequence of forms a general sequence of type $ \alpha $. Assume that $ d_1\le \ldots \le d_k $.
Let $ N_t  = \binom{n+t-1}{t} $ denote the dimension of the space of forms of degree $ t $.
We obtain the following contribution to Fröberg's conjecture \cite{froberg1985inequality}. 

\begin{theorem}\label{thm:froeberg} Let $ f_1,\ldots,f_r \in \C[x_1,\ldots,x_n]$ be a sequence of forms of type $ \alpha $. 
	Let $ \ell\in \N $. Then, the vector space $  (f_1,\ldots,f_r)_{d_k+\ell} $	has the expected dimension, so
	\begin{align*}
		h_{I}(d_k+\ell) = \alpha_{d_k} N_{\ell} + \alpha_{d_{k-1}} N_{\ell+(d_k-d_{k-1})}  + \ldots + \alpha_{d_1} N_{\ell+d_k-d_1},
	\end{align*}
	if for all $ i = 1,\ldots,k $, we have 
	\begin{align}\label{eq:m-bound-fröberg}
		\alpha_{d_1} N_{\ell+d_k-d_1} + \ldots + \alpha_{d_i} N_{\ell+d_k-d_i} + N_{\ell+d_k-d_1}(N_{\ell+d_k-d_i}-1) < N_{\ell + d_k}.
	\end{align}
Otherwise, if $ \alpha_{d_1} N_{\ell+d_k-d_1} + \ldots + \alpha_{d_k} N_{\ell} > N_{\ell + d_k} + kN_{\ell+d_k-d_1}(N_{\ell+d_k-d_1}-1)$, then the fat point scheme is empty. 
\end{theorem}
\begin{proof}
	Take $ X = \C[x_1,\ldots,x_n]_{d_1} \cup \ldots \cup \C[x_1,\ldots,x_n]_{d_k}  $, endowed with the canonical action of $ \GL_n $ on the variables. Let $ \ell \in \N_0 $ and define $ E $ to be the bundle on $ X $, embedded into $ V = \C[x_1,\ldots,x_d]_{d_k+\ell} $, which has fibers $ E_f = (f)_{d_k+\ell} $. In other words, the fiber $ E_f $ is the graded component of the principal ideal generated by $ f $ in degree $ d_k+\ell $. Note that $ E_f $ is a trivial bundle on each component of $ X $. Clearly, $ E $ is $ \GL_n $-equivariant and the Hilbert function $ h_I(d_k+\ell) $ of $ I $ at degree $ d_k+\ell $ equals the dimension of $ \langle E_{f_1},\ldots,E_{f_r} \rangle $. On the component $ \C[x_1,\ldots,x_n]_{d_i} $, $ E $ has rank $ N_{d_k-d_i+\ell} $. The claimed nondefectivity result now follows from \Cref{thm:k-factors}. The second claim is due to \Cref{thm:diagonal-filling-criterion}.
\end{proof}

We give an explicit class of examples and we plot the two bounds on a graph.
Here we fix some conditions on the forms $(f_1,\dots,f_r) \in \mathbb{C}[x_1,\dots,x_n]$. The first 20\% of the forms are of degree $d_1=5$ while the last 80\% are of degree $d_2=6$. With respect to the notation introduced above we have $\alpha_{d_1}=\lfloor 0.2r \rfloor, \alpha_{d_2}=r-\alpha_{d_1}$ and $l=1$. 

\begin{figure}[H]
\includegraphics[scale=0.5]{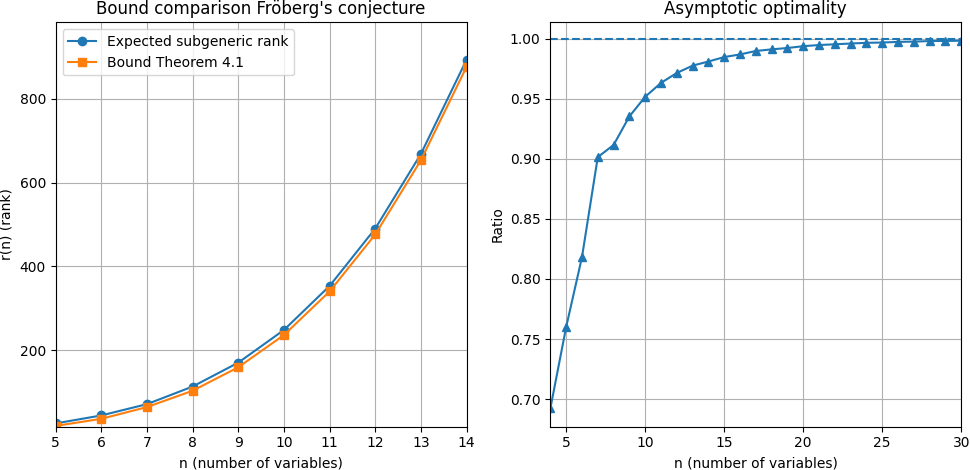}
\caption{\footnotesize{In the first figure the blue line selects the expected subgeneric rank $r$, i.e. the maximum value of $r$ such that $(f_1,\dots,f_r)_{d+2}$ is not filling. If Fröberg's conjecture is true then the graded component $(f_1,\dots,f_r)_{d+2}$ is a vector space of the expected dimension. The orange line instead plots the maximum value of $r$ such that condition \eqref{eq:m-bound-fröberg} holds true.
In the second figure we plot the ratio between the two bounds. Note that the function is asymptotically sharp, hence the blue curve has $ 1 $ as an asymptote.}}
\label{fig:froeberg}
\end{figure}

\subsection{The postulation problem for fat point schemes}
The postulation problem for fat point schemes asks the following: Given general points $ p_1,\ldots,p_r\in \mathbb C^n $ and multiplicities $ m(p_1),\ldots,m(p_r) $, what is the dimension of the linear space of all forms of degree $ D $, which vanish to order at least $ m(p_i) $ at $ p_i $. If we denote by $ I_{p}^{m} $ the ideal of forms which vanish at $ p\in \mathbb C^n $ to order at least $ m\in \N $, then we see that the postulation problem asks for the Hilbert function at degree $ D $ of the ideal $ I_{p_1}^{m(p_i)} \cap \ldots \cap I_{p_r}^{m(p_r)}$. 

This Hilbert function has an expected value, which is $ \max\{0, N_D - \sum_{i = 1}^{r} N_{m(p_i)} \} $. 
Note that we can equivalently ask for the dimension of 
\begin{align}\label{eq:fat-point-complement}
	(\langle p_1, x \rangle^{D-m(p_1)+1},\ldots, \langle p_r, x \rangle^{D-m(p_r)+1})_{D},
\end{align}
which is the orthogonal complement with respect to the apolar inner product of the space of polynomials of degree $ D $ in the fat point scheme. This gives a setting similar to Fröberg's conjecture, where we aim to understand the dimension of a graded component of an ideal. 

Note that $ r $ can be much larger than $ D $, but the multiplicities lie of course in $ \{1,\ldots,D\} $. Therefore, very often there will be several points $ p_i $ with the same multiplicity $ m(p_i) $. Suppose that there are $k$ different multiplicities, which we denote by  $m_1>\dots>m_k$, with $k \leq D$. Furthermore, denote with $\alpha_1,\dots,\alpha_k$ the number of points $ p_i $, which have multiplicity $m_1,\dots,m_k$. In particular, if we denote $\alpha=(\alpha_1,\dots,\alpha_k)$, it holds that $|\alpha|=r$. We call $ p_1,\ldots,p_r $  a fat point scheme of \emph{multiplicity type} $ \alpha $. 

The ideal from \eqref{eq:fat-point-complement} is of course not generated by general polynomials. However, it is generated by polynomials, which are general elements of a specific, $ \GL_n $-invariant subvariety: Indeed, similar to \cite{nenashev2017note}, the result on Fröberg's conjecture generalizes to settings, where $ X = \mathcal{D}_1\cup \ldots \cup \mathcal{D}_k $ is a union of $ \GL_n $-invariant classes of $ d_k $-forms $ \mathcal{D}_k $. In the previous section, we took $ \mathcal{D}_k = \C[x_1,\ldots,x_n]_{d_k} $ as the linear space of all forms of degree $ d_k $, where $ d_1,\ldots,d_k\in \N $. In this section, we are interested in the invariant classes $ \mathcal{D}_i = \{\ell^{D-m_i+1} \mid \ell \in \C[x_1,\ldots,x_n]_{1}\} $ of powers of linear forms. We obtain the following result. 

\begin{theorem}\label{thm:fat-points}
	Let $ p_1,\ldots,p_r \in \mathbb C^n$ be a fat point scheme with multiplicities $ m_1,\ldots,m_k $, of multiplicity type $ \alpha\in \N_0^k $. If for all $i=1,\dots,k$ we have
	\begin{align}\label{eq:m-bound-fat}
		\alpha_{1} N_{m_1-1} + \ldots + \alpha_{i} N_{m_i-1} + N_{m_1-1}(N_{m_i-1}-1) < N_{D}.
	\end{align}
then $ I_{p_1}^{m_1} \cap \ldots \cap I_{p_r}^{m_r} $ has the expected dimension in degree $ D $.
On the other hand, if $ \alpha_{1} N_{m_1-1} + \ldots + \alpha_{k} N_{m_k-1} - N_{m_1-1}(N_{m_k-1}-1) > N_{D}$, then we have $ (f_1,\ldots,f_r)_{d_k+\ell} = \C[x]_{d_k+\ell} $. 
\end{theorem}
\begin{proof}
	Consider the variety $ X = \mathcal{D}_1\cup \ldots \cup \mathcal{D}_r $, which is a union of powers of linear forms, as described above. 
	Note that both the dimensions and the expected dimensions of the spaces $ (I_{p_1}^{m_1} \cap \ldots \cap I_{p_r}^{m_r})_{D} $ and of $ (\langle p_1, x \rangle^{D-m_1+1},\ldots, \langle p_r, x \rangle^{D-m_r+1})_{D} $ sum up to $ N_D $. Hence, it suffices to show that the bundle $ E $, which has fibers $ E_{\ell^{D-m_i+1}} = (\ell^{D-m_i+1})_{D} $, is $ \alpha $-nondefective, where $\alpha=(\alpha_1,\dots,\alpha_k)$. The first claim now follows from \Cref{thm:k-factors} while the second one from  \Cref{thm:diagonal-filling-criterion}.
\end{proof}

As in the previous case, we plot an explicit class of examples. Here we fix $D=9, m_1=4$ and $m_2=3$. Moreover, with respect to the choice of $r$ points $p_1,\dots,p_r \in \C^n$, we impose that $\alpha_1:=\lfloor 0.3r \rfloor$ of them are of multiplicity $m_1=4$ and the remaining $\alpha_2=r-\alpha_1$ points are of multiplicity $m_2=3$. 

\begin{figure}[H]
\includegraphics[scale=0.5]{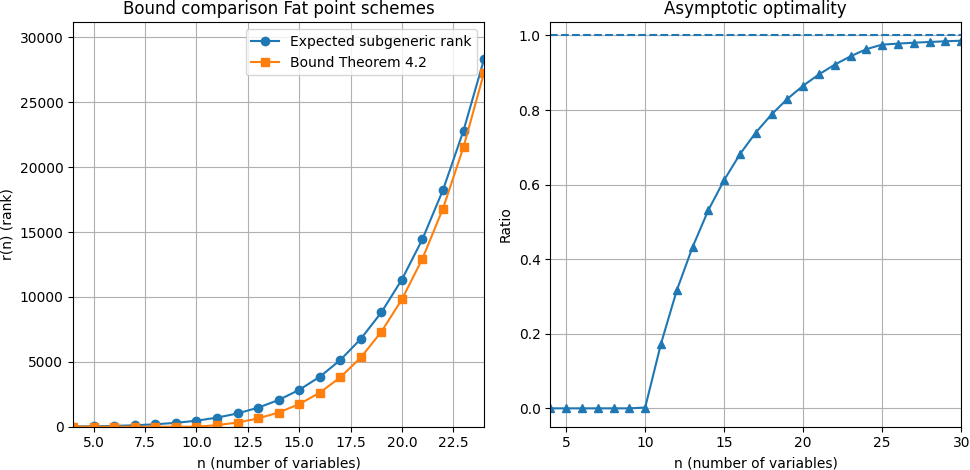}
\caption{\footnotesize{In the first figure the blue line selects the expected subgeneric rank $r$, i.e. the maximum value of $r$ such that $ (I_{p_1}^{m_1} \cap \ldots \cap I_{p_r}^{m_r})_{D} $ is not empty. The orange line plots the maximum value of $r$ such that condition \eqref{eq:m-bound-fat} holds true.
In the second figure we plot the ratio between the two bounds as before. Note that this time the bound gives a meaningful result only if the number of variables is large compared to the degree $D$.}}
\label{fig:fat-points}
\end{figure}

\subsection{Identifiability of partition rank decompositions}
Let $ d\in \N $ and $ k = (k_1,\ldots,k_l) $ a partition of $ d $, such that $ d = k_1 + \ldots + k_l $. Denote 
\begin{align*}
	\mathcal{P}_{k, \id} := \{t_1\otimes \ldots \otimes t_l \mid t_i \in (\C^n)^{\otimes k_i} \} \subseteq (\C^n)^{\otimes d}. 
\end{align*}

\noindent
We define for $ \sigma\in \mathfrak{S}_{l} $:
\begin{align*}
	\mathcal{P}_{k, \sigma} := \{t_{\sigma(1)}\otimes \ldots \otimes t_{\sigma(l)} \mid t_i \in (\C^n)^{\otimes k_i} \} \subseteq (\C^n)^{\otimes d}. 
\end{align*}

\begin{defi}
	We define the partition-rank-1 variety $ \mathcal{P}_{k}(\C^n) $ as the union of the irreducible varieties $ \mathcal{P}_{k, \sigma} $, where $ \sigma\in \mathfrak{S}_{l}  $. 
	The \emph{$ k $-partition rank} $r$ of a tensor $T$ in $(\C^n)^{\otimes d}$ is the minimum $ r\in \N_0 $ such that $ T $ can be written as a sum of $ r $ elements of  $ \mathcal{P}_{k}(\C^n) $. 
\end{defi}

\begin{theorem}\label{thm:partition-rank}
	Let $X=\bigcup_{\sigma} \mathcal{P}_{k, \sigma}$ be the partition-rank $1$ variety associated to $k=(k_1,\dots,k_l)$, a partition of $d$. The generic partition rank $ r_{g} $ with respect to the partition $ k \vdash d $ is bounded by
	\begin{align*}
		r_{g} \le \frac{n^{d}}{n^{k_1}  + \ldots + n^{k_l} - l + 1} + (n^{k_1}  + \ldots + n^{k_l} - l + 1)
	\end{align*}
	Moreover $X$ is $r$-non defective for
	\begin{align*}
		r < \frac{n^{d}}{n^{k_1}  + \ldots + n^{k_l} - l + 1} - (n^{k_1}  + \ldots + n^{k_l} - l + 1).
	\end{align*}	
\end{theorem}
\begin{proof}
	$ (\C^n)^{\otimes d} $ has dimension $ n^d $. The irreducible components of $ \mathcal{P}_{k} $ are all equidimensional of dimension $ n^{k_1}  + \ldots + n^{k_l} - l + 1 $. From \Cref{cor:reducible-secants-nondefective}, we obtain the desired result. 
\end{proof}

\begin{remark}
	Our setting fixes a partition and considers all permutations associated with that partition. It is of course possible to consider several partitions at once. In that case, the irreducible components of $ X $ are not equidimensional and one obtains a more complicated bound, similar to \Cref{thm:froeberg} and \Cref{thm:fat-points}. Naslund \cite{Naslund_2020} originally considered partition ranks as a generalization of slice rank. Slice rank is the special case corresponding to the partition $ k = (1, d-1) $. Note that our bound for nondefectivity is vacuous in the case of slice rank. This is expected, since secants to the slice rank-$ 1 $ variety are defective. 
\end{remark}

\subsection{Mixture distributions}

A common problem in applications is to identify the parameters of a mixture distribution from its moments. This is especially important for Gaussian mixtures, see the explanations in (\cite{Blomenhofer_Casarotti_Michalek_Oneto_2022}, \cite{Taveira_Blomenhofer_GM6_2025}), but also for other types of distributions. The results from \cite{Blomenhofer_Casarotti_2023} and \cite{Taveira_Blomenhofer_GM6_2025} guarantee the identifiability of Gaussian mixtures, and, in principle, generalize to mixtures of other types of $ \GL_n $-invariant distributions. 

However, it was previously unknown what happens if one desires two different types of mixtures. As one concrete example, one might ask if a mixture of 5 Gaussian distributions $ \mathcal{N}(\mu_1, \Sigma_1),\ldots,\mathcal{N}(\mu_5, \Sigma_5) $ and 7 Laplace distributions $ L(\mu_6, \Sigma_6),\ldots,L(\mu_{12}, \Sigma_{12}) $ is \emph{algebraically identifiable}, i.e., if for general $ \mu_i, \Sigma_i $, the moments of degree, say $ 5 $, determine the parameters up to finitely many possibilities. 
With our \Cref{thm:2-factors} and \Cref{thm:k-factors}, we can contribute to this problem. We first need a bit of background on probability theory. 
If $ Y $ is a (real-valued) random vector on $ \R^n $, then its characteristic function is defined as 
\begin{align*}
	\varphi_{Y}(t) = \E[\exp(i t^{T}Y)], \qquad (t\in \mathbb C^n)
\end{align*}
It is well-known that the characteristic function uniquely determines its probability distribution. The characteristic function can of course be viewed as a power series in variables $ t_1,\ldots,t_n $. 
The \emph{moment forms} $ \mathcal{M}_d(Y) \in \C[t_1,\ldots,t_n]_{d}$ of $ Y $ are defined as the $ d $-homogeneous parts of $ \varphi_{Y}(-it) $, up to rescaling as follows:
\begin{align*}
	\mathcal{M}_d(Y) = \frac{1}{d!} [\varphi_{Y}(-it)]_{d}
\end{align*}
Here, we denote by $ [f]_{d} $ the $ d $-homogeneous part of a power series. Some parameterized families of probability distributions have nice expressions for the moment forms. E.g., if $ Y\sim \mathcal{N}(\mu, \Sigma) $ is Gaussian distributed with mean $ \mu\in \R^n $ and covariance matrix $ \Sigma \in \R^{n\times n} $, then $ \varphi_{Y}(t) = \exp(\i t^{T}\mu - \frac{1}{2} t^{T}\Sigma t) $. 
On the other hand, if $ Z\sim L(\mu, \Sigma) $ is symmetric multivariate Laplacian distributed with mean $ \mu $ and scale matrix $ \Sigma $, then the characteristic function is given by 
\begin{align*}
	\varphi_{Y}(t) = \frac{\exp\bigl(\i \mu^T t\bigr)}{1 + \tfrac{1}{2}\, t^T \Sigma t}, 
\end{align*}
Abbreviating $ q_{\Sigma} = t^{T}\Sigma t $ and $ \ell_{\mu} = t^{T}\mu $, this yields the expressions 
\begin{align*}
	\mathcal{M}_5(\mathcal{N}(\mu, \Sigma)) &= \ell_{\mu}^5 + 10q_{\Sigma}\ell_{\mu}^3 + 15q_{\Sigma}^2\ell_{\mu}, \\
	\mathcal{M}_5(L(\mu, \Sigma)) &=\ell_{\mu}^5 + 10q_{\Sigma}\ell_{\mu}^3 + 30q_{\Sigma}^2\ell_{\mu}.
\end{align*} 
The Gaussian moment variety $ \GM_5(\mathbb C^n) $ is defined as the Zariski closure of
\begin{align*}
	\{\ell_{\mu}^5 + 10q_{\Sigma}\ell_{\mu}^3 + 15q_{\Sigma}^2\ell_{\mu} \mid \mu\in \mathbb C^n, \Sigma\in \mathbb C^{n\times n}\}.
\end{align*}
Likewise, one may define the Laplacian moment variety $ \mathrm{LM}_{5}(\mathbb C^n) $. 
Note that in applications, one is of course only interested in real points of $ \GM_5(\mathbb C^n) $, where $ \Sigma $ is positive definite. However, algebraic identifiability of the complex variety implies algebraic identifiability of its real points, since the real points form a Zariski dense subset of $ \GM_5(\mathbb C^n) $.

\begin{theorem}\label{thm:mixture-gaussian-laplace}
	A mixture of $ r $ Gaussians $ (\mathcal{N}(\mu_i, \Sigma_i))_{i = 1,\ldots,r} $ and $ s $ Laplace distributions $ (L(\mu_{r+i}, \Sigma_{r+i}))_{i = 1,\ldots,s} $ is algebraically identifiable from degree 5 moments, if 
	$ (r+s) \le \dfrac{\binom{n+4}{5}}{\binom{n+1}{2} + n} - (\binom{n+1}{2} + n) $. 
\end{theorem}
\begin{proof}
	The ring of $ d $-forms in $ x_1,\ldots,x_n $ is an irreducible $ \GL_n $-module and both $ X_1 = \GM_5(\mathbb C^n) $ and $ X_2 = \mathrm{LM}_5(\mathbb C^n) $ are $ \GL_n $-invariant varieties. Hence, we may apply \Cref{thm:2-factors} with the equidimensional variety $ X = (X_1)_{\smooth} \sqcup (X_2)_{\smooth} $ and $ E = T $ the tangent bundle. 
\end{proof}

\begin{remark}
	As a small technical detail, note that in the proof of \Cref{thm:mixture-gaussian-laplace}, the sets $ (X_1)_{\smooth} $ and $ (X_2)_{\smooth} $ do not intersect, so the disjoint union is unnecessary. Indeed, if we had $ f:= \mathcal{M}_5(\mathcal{N}(\mu, \Sigma)) = \mathcal{M}_5(L(\nu, S)) $ for some parameters, then both $ \ell_{\mu} $ and $ \ell_{\nu} $ would be the unique linear form dividing $ f $ and must thus be equal. Using $ \ell_{\mu} = \ell_{\nu} $, the identity simplifies to $ 10q_{\Sigma}\ell_{\mu}^2 + 15q_{\Sigma}^2 = 10q_{S}\ell_{\mu}^2 + 30q_{S}^2 $. By evaluating on the vanishing locus of $ \ell_{\mu} $, we see that $ q_{\Sigma} \equiv \pm 2q_S $ modulo $ \ell_{\mu} $. A short check shows that this is only possible, if both $ q_{\Sigma}$ and $ q_{S} $ are multiples of $ \ell_{\mu}^2 $. 
\end{remark}

Similar results may of course proved for other parameterized types of mixture distributions, provided their moment forms have a suitable description in terms of the parameters and the map from the parameter space to the moment forms is $ \GL_n $-equivariant. 
For the concrete example mentioned above, we obtain that a mixture of 5 Gaussians and 7 Laplace distributions is algebraically identifiable from degree-$ 5 $ moments, if $ n\ge 27 $. Note that we cannot have algebraic identifiability for any degree strictly smaller than $ 5 $, since the Gaussian moment variety $ \GM_4(\mathbb C^n) $ is already $ 2 $-defective \cite{Taveira_Blomenhofer_GM6_2025}. 

\medskip
\paragraph{Asymptotic optimality}
The result in \Cref{thm:mixture-gaussian-laplace} is asymptotically optimal. Indeed, algebraic identifiability can for dimension reasons only hold, if $ (r+s) \le \dfrac{\binom{n+4}{5}}{\binom{n+1}{2} + n} $. The additive term $ -\binom{n+1}{2} - n $ on the right hand side in \Cref{thm:mixture-gaussian-laplace} is of order $ \mathcal{O}(n^2) $ and thus asymptotically dominated by the left term, which is in $ \mathcal{O}(n^3) $. As discussed in \cite{Blomenhofer_Casarotti_2023}, in the setting of $ \GL_n $-invariant subvarieties of tensor spaces, stationarity arguments often tend to produce bounds, which are asymptotically optimal for large values of $ n $. See also the examples in \cite{Blomenhofer_Casarotti_2023}.

	\subsection*{Acknowledgements}
	Alexander Taveira Blomenhofer is supported by the ERC grant of Matthias Christandl under Agreement 818761.
	Alex Casarotti is supported by the PRIN 2022 (Birational geometry of moduli spaces and special varieties) grant of Alex Massarenti. 
	Part of this work was done while Alexander Blomenhofer was visiting the University of Ferrara. We thank Alex Massarenti for inviting and funding the visit via his grant.

	\bibliography{bibML}
	\bibliographystyle{plain}
\end{document}